\documentclass[A4paper,12pt]{article}
\usepackage{latexsym}
\usepackage{mathrsfs}
\usepackage{amssymb}
\usepackage{amscd}
\usepackage[dvips]{graphicx}                  

\newtheorem{theorem}{Theorem}

\newenvironment{proof}[1][Proof]{\textbf{#1.} }{\ \rule{0.5em}{0.5em}}
\long\def\symbolfootnote[#1]#2{\begingroup%
	\def\thefootnote{$\;$}\footnote[#1]{$^*$#2}\endgroup}

	\title{Remarks on the existence of measurable selectors}
\author{Joanna Jureczko \footnote{The author was partially supported by Wroclaw University of Science and Technology grant no. 8211104160.
		}}

\begin{document}
	
	\maketitle
	
	\begin{abstract}
		A classical theorem from measure theory that gives a sufficient condition for a multifunction to have a measurable selection is Kuratowski and Ryll-Nardzewski Selection Theorem. The aim of this paper is to show some generalizations of this result. 
	\end{abstract}
	\section{Introduction}
	
	Undoubtedly, a classical theorem from measure theory which gives a sufficient condition for a multifunction to have a measurable selection is Kuratowski and Ryll-Nardzewski Selection Theorem, \cite{KRN}, (see also \cite{KM}).
	This theorem is a significant result of so-called theory of selectors closely related  to the axiom of choice and its various formulations.
	
	In the literature, there are different types of problems of the existence of selectors corresponding to different formulations of the axiom of choice. These formulations are equivalent on the basis of set theory, however, the topological issues to which they conduct are generally significantly different.
	
	The problem of the existence of measurable selectors has been  considered by researchers working in different parts of human knowledge and it is enough to observe that the  Kuratowski and Ryll-Nardzewski paper, \cite{KRN}, has been cited about 200 times in top cited scientific journals (not only mathematical ones).
	
	As the starting point, one has designated the following papers \cite{KP, FHJ, FGPR} and indirectly related
	\cite{F, AF, H, SZ}. 
	Papers \cite{KP, H} concern Borel sets of metric spaces but \cite{FGPR} concerns investigations of different assumptions of spaces. The paper  \cite{AF} is  mentioned here because it is the continuation and extension of \cite{FGPR} but papers \cite{H, SZ} are extensions of \cite{KP}.
	
	In paper \cite{F}  the author investigates the techniques which can be applied to the problem of measurability sets of different properties and connections between these properties and various sufficient conditions for spaces having these properties. The author gives also some generalizations of well-known theorems on selectors and formulates some open problems.
	
	In paper \cite{FHJ} the authors consider the question: If $\mathcal{A} = \{A_\alpha: \alpha < \kappa\}$ is a point finite family, (i.e. every $x \in X$ is in finitely many $A \in \mathcal{A}$) of subsets of a metric space $X$ such that  $\bigcup \mathcal{A}'$ 
	is Borel in $X$ for every $\mathcal{A}' \subseteq \mathcal{A},$ then there is a fixed $\beta < \omega_1$ such that every $\bigcup \mathcal{A}'$  
	is of Borel class $< \beta$. 
	As it turns out, this problem is related to  the question of A. H. Stone: if $f\colon X \to Y$ is a Borel isomorphism between metric
	spaces ($f$ is a bijection and $f$ and $f^{-1}$
	take open sets to Borel sets), then there is $\beta < \omega_1$ such that for all open $U \subseteq Y,$ $f^{-1}[U]$ is of Borel class $< \beta$ and what is more important for further research of the author of this paper to the existence of large cardinals, (i.e. measurable numbers), (see e.g. \cite{FD, TJ}). 
	
	The results presented in \cite{KP} are considered for analytic sets.
	\\We say that a set $A$ in a metric space $X$ is \textit{analytic} in $X$ if $A$ is the result of an operation $\mathcal{A}$ on closed sets in $X$. A topological space is \textit{absolutely analytic} iff it is homeomorphic to an analytic set in a complete metric space. 
	If $X$ and $Y$ are metric spaces, and $K(Y)$ is the family of non-empty compact subsets of $Y$, then $F\colon X\to K(Y)$ is \textit{of lower class $\alpha$} (where $\alpha$ is a countable ordinal) iff $\{x\colon F(x)\cap V\not = \emptyset\}$ is a Borel set of additive class $\alpha$ in $X$ whenever $V$ is open in $Y$. A function $f\colon X\to Y$ is \textit{of class $\alpha$} iff $f^{-1}(V)$ is a Borel set of additive class $\alpha$ in $X$ whenever $V$ is open in~$Y$. 
	\\The main result of \cite{KP} is the following theorem.
	\begin{theorem}[\cite{KP}]
		If $X$ is absolutely analytic, $Y$ is metric and $F\colon X \to K(Y)$ is of lower class $\alpha>0$, then there is a function $f\colon X\to Y$ of additive class $\alpha$ such that $f$ is a selector for $F$ (i.e., $f(x)\in F(x)$ for all $x\in X)$.
	\end{theorem} 
	It is worth noticing that the proof of the above theorem relies on  techniques used by Kuratowski and Ryll-Nardzewski in \cite{KRN}.
	
	The main results  of \cite{FGPR} concern the existence of a selector in a few essential cases (see Section 2.8 and 2.9).
	
	The aim of this paper is to prove that in the main results of \cite{FGPR} the assumption of cardinality of point-finite families and/or the image of lower semi-continuous functions can be omitted.

	\section{Definitions and previous results}
	
	\textbf{2.1.} Let $F \colon X \to \mathcal{P}(X)$ be a mapping. Any function $f \colon X \to Y$ such that $$f(x)\in F(x) \textrm{ for each } x \in X$$
	is called \textit{a selector for F}. 
	
	Let $\mathcal{R}$ be a field of subsets of $X$. Denote by $\mathcal{R}_\sigma$  a $\sigma$-field generated by $\mathcal{R}$. Let $\mathcal{S} \subset \mathcal{P}(X)$  be a $\sigma$-field. A function $f \colon X \to Y$ is \textit{$\mathcal{S}$-measurable} iff 
	$$f^{-1}(U) \in \mathcal{S} \textrm{ for each open } U \subset Y.$$
	A function $F \colon X\to \mathcal{P}(Y)$ is \textit{lower-$\mathcal{S}$} iff 
	$$\{x \in X \colon F(x)\cap U \not = \emptyset\}\in S \textrm{ for each open } U \subset Y$$
	and \textit{upper-$\mathcal{S}$} iff 
	$$\{x \in X \colon F(x)\cap A \not = \emptyset\}\in S \textrm{ for each closed } A \subset Y.$$
	If $\mathcal{S}$ is a family of all open subsets of $X$ and $F$ is lower-$\mathcal{S}$ then we say that $F$ is \textit{lower semi-continuous}. Respectively one defines $F$  \textit{upper semi-continuous} function.
	\\\\
	\textbf{2.2.} Let $X$ be a topological space. By $Cl^+(X)$  is denoted the family of all non empty closed subsets of $X$ and by $K^+(X)$ is denoted the family of all non-empty compact subsets of $X$.
	\\\\
	\textbf{2.3.} A family $\mathcal{A}$ of subsets of $X$ is \textit{point-finite} iff for each $x \in X$
	the set $\{A\ \in \mathcal{A} \colon x \in A\}$ is finite.
	\\\\
	\textbf{2.4.}  A family $\mathcal{A}$ of sets is \textit{compact} iff for each sequence $(A_i)_{i \in \omega} \subset \mathcal{A}$ if  $\bigcap_{i=0}^{m}A_i \not = \emptyset$ for each $m \in \omega$ then $\bigcap_{i=0}^{\infty}A_i \not = \emptyset$.
	\\\\
	\textbf{2.5.} Let $\mathcal{S}$ be a $\sigma$-field of subsets of $X$ and let $\mu$ be a measure defined on $\mathcal{S}$. All sets of $\mathcal{S}$ are called \textit{measurable}. By $(X, \mathcal{S},\mu)$ is denoted the space $X$ endowed with the measure structure. All measures considered in this paper are $\sigma$-additive and finite.
	\\\\
	\textbf{2.6.} A measure $\mu$ is \textit{perfect} iff for every $\mathcal{S}$-measurable function $f \colon X \to \mathbb{R}$ and $E \subset \mathbb{R}$ there exists a Borel set $B \subset E$ such that $\mu(f^{-1}(E)) = \mu(f^{-1}(B))$.
	\\\\Sazonov in \cite{VS} proved the following result. (Notice that the second part of this result is exactly the definition of so-called quasi-compact measure given by Ryll-Nardzewski in \cite{CRN}).
	\\\\
	\textbf{Fact 1.} A measure $(X, \mathcal{S},\mu)$ is perfect iff for each sequence of measurable sets $(E_i)_{i \in \omega}$and for each $\varepsilon > 0$ there exists a set $E \in \mathcal{S}$ such that $\mu(E)>\mu(X)-\varepsilon$ and the sequence $(E \cap E_i)_{i\in \omega}$ is a compact family.  
	\\\\
	\textbf{2.7.} 
	The following result is given by Kuratowski and Ryll-Nardzewski, \cite[p. 401]{KM}
	\\\\
	\textbf{Fact 2.} \textbf{[Kuratowski and Ryll-Nardzewski, (\cite{KRN})]} Let $Y$ be a Polish space, let $\mathcal{A}$ be a field of subsets of $X$, (no assumptions are made about topology of $X$). If $\mathcal{R}$ is a $\sigma$-field generated by $\mathcal{A}$ and a mapping $F \colon X \to P(Y)$ is lower-$\mathcal{R}$ then $F$ admits an $\mathcal{R}$-measurable selector $f$.
	\\\\
	\textbf{2.8.} The next three facts are the main results of \cite{FGPR}.
	\\\\
	\textbf{Fact 3.}  	Let $(X, \mathcal{S}, \mu)$ be a space with perfect measure. Let $\mathcal{B}$ be a point-finite cover consisting of sets of measure zero and such that $\mathcal{B}$ is less than the least measurable cardinal. Then there exists $\mathcal{B}' \subset \mathcal{B}$ such that $\bigcup\mathcal{B}'$ not measurable  (i.e. $\bigcup\mathcal{B}' \not \in \mathcal{S}$) and has inner measure zero.
	\\\\
	\textbf{Fact 4.}If $(X, \mathcal{S}, \mu)$ is a  metrizable space with a complete perfect  measure, $Y$ is a complete metric space,  $F \colon X \to Cl^+(Y)$ is lower semi-continuous and $|F(X)|$ is less than the least measurable cardinal then $F$ admits an $\mathcal{S}$-measurable selector.
	\\\\
	\textbf{Fact 5.}If $(X, \mathcal{S}, \mu)$ is a space with a complete perfect measure, $Y$ is a metric space, $F \colon X \to K^+(Y)$ is lower $\mathcal{S}$-measurable and $|F(X)|$ is less than the least measurable cardinal then $F$ admits a measurable selector.
	\\\\
	\textbf{2.8.} A \textit{pseudobase} for a space $X$ is a family $\mathcal{C}$ of non-empty subsets of $X$ such that for each non-empty subset $V \subseteq X$ there exists $C \in \mathcal{C}$ such that $ C \subseteq V$. The smallest cardinal of the pseudobase of $X$ is the \textit{pseudoweight} of $X$ and is denoted $\pi w(X)$.
	\\\\
	\textbf{Fact 6.\cite{FG}} Let $\mathcal{B}$ be a point-finite family of meager sets covering the pseudobasically compact space $X$ and $\pi w(X) \leqslant 2^\omega$. Then there exist a family $\mathcal{B}'\subset \mathcal{B}$ such that $\bigcup_{\mathcal{B}'}$ has not the Baire property.
	\\\\
	\textbf{Fact 7.\cite{FGPR}} If $X$ is a pseudobasically compact space, $\pi w(X) \leqslant 2^\omega$, $Y$ is a metric space and $F \colon X \to K^+(Y)$ is lower or upper Baire measurable then $F$ admits a measurable selector.
	
	\section{Main results}
	
	\begin{theorem}
		Let $(X, \mathcal{S}, \mu)$ be a space with perfect measure. Let $\mathcal{B}$ be a point-finite cover consisting of sets of measure zero. Then there exists $\mathcal{B}' \subset \mathcal{B}$ such that $\bigcup\mathcal{B}'$ not measurable  (i.e. $\bigcup\mathcal{B}' \not \in \mathcal{S}$) and has inner measure zero. 
	\end{theorem}
	
	\begin{proof}
		Suppose, in contrario, that all $\bigcup \mathcal{B}'$ are measurable for all $\mathcal{B}'\subset \mathcal{B}$. Consider the sets
		$$W_n = \{x \in X \colon |\{B \in \mathcal{B} \colon x \in B\}| = n\},$$
		where $n \in \omega$. There exists $n \in \omega$ such that $\mu(W_n) > 0$. 
		\\Let $n_0 = \min\{n \in \omega \colon \mu(W(n)) >0\}$. 
		Enumerate $$\mathcal{B} = \{B_i \colon i \in I\}.$$ 
		Let $V_x = \{i\in I \colon x \in B_i\}$. By our assumption, $|V_{x}|< \omega$ for any $x \in X$.
		For any $J \subset I$ consider
		$$W_{n_0}(J) = \{x \in X \colon V_x\cap J \not = \emptyset\}$$
		Now, construct a countable sequence of elements $(W^k_{n_0}(J))_{k \in \omega}$ having the following properties
		\begin{itemize}
			\item[(1)] $W^0_{n_0}(J) = W_{n_0}(J)$,
			\item[(2)] $W^{k+1}_{n_0}(J) \subset W^k_{n_0}(J)$ for any $k \in \omega$,
			\item[(3)] $\mu(W^{k}_{n_0}(J)) >0$ for any $k \in \omega$,
			\item[(4)] $\mu(W^{k+1}_{n_0}(J)) < \frac{1}{2^n} \mu(W^{k}_{n_0}(J))$ for any $k \in \omega$.	
		\end{itemize}
		
		Let
		$$\mathcal{B}_{n_0}^k(J) = \{B_i \in \mathcal{B} \colon  \exists_{x \in W_{n_0}^k(J)}\  x \in B_i\}$$
		By $(2)$ we conclude $ \mu(\bigcap_{k=0}^{\infty}W^k_{n_0}(J)) = 0$.
		By $(1)-(4)$ for any $m=0,1,...$ we have $\bigcap_{k=0}^{m}W^{k}_{n_0}(J) \not = \emptyset$  and hence $\bigcap_{k=0}^{m}(\bigcup \mathcal{B}^k_{n_0}(J)) \not = \emptyset$. 
		However, there exists $J\subset I$ such that 
		$\bigcap_{k=0}^{\infty} \bigcup \mathcal{B}^k_{n_0}(J_0) = \emptyset$.  Hence for no $\varepsilon>0$ and no $E \in \mathcal{S}$ such that $\mu(E) > \mu(X)-\varepsilon$ and $\bigcap_{k=0}^{m} (\bigcup E\cap \mathcal{B}^k_{n_0}(J_0)) \not = \emptyset$ there is  $\bigcap_{k=0}^{\infty} (\bigcup E \cap\mathcal{B}^k_{n_0}(J_0) ) \not = \emptyset$. By Fact 1 we have that $(X, \mathcal{S}, \mu)$ is not perfect. A contradiction.	
	\end{proof}
	
	\begin{theorem}
		Let $(X, \mathcal{S},\mu)$ be a metrizable space with a complete perfect measure. Let $Y$ be a complete metric  space. If $F\colon X \to Cl^+(Y)$ is lower semi-continuous, then $F$ admits an $\mathcal{S}$-measurable selector. 
	\end{theorem}
	
	\begin{proof} (We follow the method of \cite{FGPR}).
		We will define a separable $T \subset Y$ such that 
		$$\mu(X\setminus \{x \in X \colon F(x)\cap T \not = \emptyset\}) = 0.$$
		Then, Fact 2 there is an $\mathcal{S}$-measurable selector $\overline{f}$ for $F|\{x \in X \colon F(x)\cap T \not = \emptyset\}$. Our required selector $f$ will be defined as follows
		$$
		f(x) = \left\{ \begin{array}{ll}
		\overline{f}(x) & \textrm{iff $F(x)\cap T \not = \emptyset$}\\
		y & \textrm{iff $F(x)\cap T = 0$}
		\end{array} \right.
		$$
		where $y\in F(x)$ is any element.
		Now, we construct the required $T$. Let $$F(X) = \{F_\alpha \colon \alpha < |F(X)| \colon F_\alpha \not = F_\beta \textrm{ for }  \alpha \not = \beta\}.$$
		By induction on $n$ we construct a sequence $\{\mathcal{U}_n \colon n \in \omega\}$   such that
		\begin{itemize}
			\item[(1)] $\mathcal{U}_n$ is a  locally finite open cover $Y$ 
			\item[(2)] $\forall_{n \in \omega}\  \mathcal{U}_n = \{U \in Y \colon diam(U) < \frac{1}{n}\}$
			\item[(3)] $\forall_{\alpha < |F(X)|}\  \forall_{n \in \omega}\  \mathcal{U}^n_{\alpha} \in \mathcal{U}_n$  and $U^n_\alpha \cap F_\alpha \not = \emptyset$,
			\item[(4)] $G^n_{\alpha} = \{x \in X \colon F(x)\cap U^n_\alpha \not = \emptyset\},$
			\item[(5)] $H_\alpha^n = G^n_{\alpha} \setminus \bigcup \{G^n_{\beta} \colon \beta < \alpha\}$. 
		\end{itemize}
		Notice that the set in $(5)$ is a  cover of $X$ consisting of disjoint $F_\sigma$ sets.	
		By Theorem 1, there exists a countable $A_n \subset |F(X)|$  such that $\mu(H^n_{\alpha}) >0$ for all $\alpha  \in A_n$  and $\mu(X \setminus \bigcup \{H^n_\alpha \colon \alpha \in A_n\}) =0$.  (It is enough to consider $\mathcal{U}_{n+1}$ as a  locally finite open cover $\bigcup \{U^n_\alpha \colon \alpha \in A_n\}$ instead of (1)).
		Now, take 
		$$T = \bigcap \{\bigcup\{U^n_\alpha \colon \alpha \in A_n\}\} \colon n \in \omega$$ is the required countable set.
	\end{proof}
	
	\begin{theorem}
		Let $(X, \mathcal{S}, \mu)$ be a space with a complete perfect measure and let $Y$ be a metric space. Let $F \colon X \to K^+(Y)$ be a lower $\mathcal{S}$-measurable function. Then $F$ admits a measurable selector. 
	\end{theorem}
	
	\begin{proof}
		We will follow the method used in the proof of Fact 2 in \cite{KM} and \cite{FGPR}.
		Without the loss of generality, we can assume that $Y$ is a completion of $\bigcup\{F(x) \colon x \in X\}$. 
		
		Let $\rho$ be a metric on $Y$ such that $diam(Y)<1$. We will define a selector $f \colon X \to Y$ as the limit of functions $f_n \colon X \to Y$ such that 
		\begin{itemize}
			\item[(1)] $f_n$ is $\mathcal{S}$-measurable,
			\item[(2)] $\forall_{x \in X}\ \rho(f_n(x), F(x)) < \frac{1}{2^k}$,
			\item[(3)] $\forall_{x \in X}\ \forall_{n>0}\ \rho(f_n(x), f_{n+1}(x)) < \frac{1}{2^{k+2}}$,
			\item[(4)] $f_0$ is a constant function.   
		\end{itemize}

		Suppose that the sequence $\{f_k \colon  k < n\}$ fulfilling $(1)-(4)$ has been constructed. We construct $n$-step.
		Let $$\mathcal{U}_n = \{U \subset Y \colon diam(U) < \frac{1}{2^{n+1}}\}$$
		be any locally finite cover of $Y$. Since $F(x)$ is compact, $$|\{U \in \mathcal{U}_n \colon F(x)\cap U \not = \emptyset\}|<\omega.$$
		
		Well order
		$\mathcal{U}_n = \{U^n_\alpha \colon \alpha \in T_n\}$.
		For any $t \in T_n$ let 
		$$C^n_\alpha = \{x \in X \colon F(x)\cap U^n_\alpha \not = \emptyset\}$$
		and
		$$D^n_\alpha = \{x \in X \colon \rho(f_{n-1}(x), y) < \frac{1}{2^{n+1}} \textrm{ for some } y \in U^n_\alpha\}.$$
		Obviously, $\{C^n_\alpha \colon \alpha \in T_n\}$ and $\{B^n_\alpha \colon \alpha \in T_n, B^n_\alpha = C^n_\alpha \cap D^n_\alpha\}$
		are point-finite covers of $X$.
		Moreover,  $\bigcup\{B^n_\alpha \colon \alpha < T'_n\} \in \mathcal{S}$ for any $T'_n \subset T_n$.
		By Theorem 1 there is a countable $S'_n \subset T'_n$ such that $\mu(B^n_\alpha)=0$ for all $\alpha \in S'_n$ and
		$$\mu(\bigcup\{B^n_\alpha \colon \alpha \in T'_n \setminus S'_n\})=0.$$
		Now, put
		$$
		A^n_\alpha = \left\{ \begin{array}{ll}
		B^n_\alpha \setminus (\bigcup\{B^n_\beta \colon \beta< \alpha, \beta \in T'_n\setminus S'_n\}\cup \bigcup\{B^n_\gamma \colon \gamma \in S'_n\}) & \textrm{iff } \alpha \in T'_n\setminus S'_n\\
		B^n_\alpha \setminus \bigcup\{B^n_\beta \colon \beta< \alpha, \beta \in T'_n\} & \textrm{iff } \alpha \in S'_n
		\end{array} \right.
		$$
		Thus $\{A^n_\alpha \colon \alpha \in T_n\}$ is a disjoint cover of $\mathcal{U}_n$
		and
		$\bigcup\{A^n_\alpha \colon \alpha \in T'_n\}\in \mathcal{S}$ for each $T'_n \subset T_n$.
		
		Now, put $f_n(x)=z_\alpha$ iff $x\in A^n_\alpha$ and $z_\alpha \in \mathcal{U}^n_\alpha$. Clearly, $f_n$ fulfills $(1)-(4)$. Moreover, for all $x \in X$ 
		$$ f(x) =\lim_{n\to\infty} f_n(x) \textrm{ and } f(x) \in F(x)$$
		and the sequence is uniformly convergent to $f$.
		By Fact 2, the selector $f$ is $\mathcal{S}$-measurable.  
	\end{proof}

\noindent
{\sc Joanna Jureczko}
\\
Wroc\l{}aw University of Science and Technology, Wroc\l{}aw, Poland
\\
{\sl e-mail: joanna.jureczko@pwr.edu.pl}

\end{document}